\documentclass[a4paper]{amsart}
\usepackage{graphicx}
\usepackage{amssymb}
\usepackage{amsmath}
\usepackage{amsthm}
\usepackage{amscd,url}
\usepackage[all,2cell]{xy}

\UseAllTwocells \SilentMatrices
\newtheorem{thm}{Theorem}[section]

\newtheorem{cor}[thm]{Corollary}

\newtheorem{lem}[thm]{Lemma}
\newtheorem{exm}[thm]{Example}

\newtheorem{prop}[thm]{Proposition}

\theoremstyle{definition}
\newtheorem{defn}[thm]{Definition}
\theoremstyle{remark}

\numberwithin{equation}{section}

\begin{document}
\title[$\tau$-tilting modules, depth and  delooping level]{$\tau$-tilting modules, depth and delooping level}
\author[Xu and Zhang] {Mingfei Xu and Xiaojin Zhang$^*$}

\makeatletter
\@namedef{subjclassname@2020}{\textup{2020} Mathematics Subject Classification}
\makeatother

\thanks{$^*$ The corresponding author}
\thanks{Email: mfxu123@163.com, xjzhang@jsnu.edu.cn}

\maketitle

\dedicatory{}%
\commby{}%
\begin{abstract}
Let $A$ be a finite-dimensional basic algebra over an algebraically closed field $K$, $T$ a finitely generated support $\tau$-tilting right $A$-module and $B={\rm End}_A T$. Denote by ${\rm Fac}T$ the subcategory of finitely generated right $A$-modules generated by $T$. We define the depth relative to $T$ and the delooping level relative to $T$ and show that the finitistic dimension of the opposite algebra of $B$ is bounded by the depth of $\textup{Fac}T$ relative to $T$ and the delooping level of $\textup{Fac}T$ relative to $T$ whenever $T$ is self-orthogonal. We give applications to the finitistic dimension conjecture. More precisely, we show that if $A$ is an algebra of finite representation type, then the finitistic dimension of $B^{op}$ is finite.
\end{abstract}

\section{Introduction}

In 2014, Adachi, Iyama and Reiten \cite{AIR} introduced $\tau$-tilting theory as a generalization of tilting theory from the viewpoint of mutation. From then on, $\tau$-tilting theory became popular in the representation theory of finite-dimensional algebras. In the $\tau$-tilting theory, support $\tau$-tilting modules are the most important objects. For support $\tau$-tilting modules over special algebras, we refer to \cite{A, M, DIJ, IZ} and the references therein. For the infinitely generated case, that is, the silting modules, we refer to \cite{AnMV}. For the connection between $\tau$-tilting modules and homological conjectures, we refer to \cite{CLZZ, Z}.

For a finite-dimensional algebra $A$, the finitistic dimension of $A$ is the supremum of the finite projective dimension of finitely generated modules. It is conjectured that the finitistic dimension of $A$ is always finite. This is the famous finitistic dimension conjecture which is still open now. For the study of the finitistic conjecture, we refer to \cite{GPS, IT, K, PX, Xi2} and the references therein.  In 2022, Gelinas \cite{G} introduced a new homological notion delooping level to study the finitistic dimension. And he proved the following theorem \cite[Proposition 1.3]{G}.

 \begin{thm}\label{1.1} For a finite-dimensional algebra $A$, the finitistic dimension of $A^{op}$ is bounded by the depth of $A$ and the delooping level of $A$. That is, the depth of $A$ $\leq$ the finitistic dimension of $A^{op}$ $\leq$ the delooping level of $A$.
 \end{thm}

 If the delooping level of $A$ is finite for any algebra $A$, then by Theorem \ref{1.1} the finitistic dimension conjecture is proved. Unfortunately,  Kershaw and Rickard \cite{KR} constructed a counter-example for this. Moreover, Sen \cite{S} computed the delooping level for Nakayama algebras. In 2025, Guo and Igusa \cite{GuI} introduced the definitions of $k$-delooping level and derived delooping level and gave a more precise bound for the finitistic dimension. Later Guo \cite{Gu} showed that the derived delooping level is left and right symmetric. Chen and Hu \cite{CH} studied delooping level and derived delooping level in terms of functions. Recently, Chen \cite{C} showed that the delooping level is not preserved under derived equivalence.

Since $A$ is always a support $\tau$-tilting module in $\text{mod}A$, it is natural to ask: Is there a generalized version of Theorem \ref{1.1} in terms of support $\tau$-tilting modules? In this paper, we give a positive answer to the question and then apply our result to the finitistic dimension conjecture.

In what follows, let $A$ be a finite-dimensional algebra over an algebraically closed field $K$ and let $T\in \text{mod} A$ be a support $\tau$-tilting module with $B=\text{End}_A T$. Denote by $\mathcal{C}=\text{Fac}T=T(S)=\text{Filt(Fac} S)$, where $S=S_1\bigoplus \cdots \bigoplus S_m$ is the left finite semi-brick in \cite[Theorem 2.3(2)]{As} and $S_i$ is a brick. We first define the depth of $\mathcal{C}$ relative to $T$ (See Definition \ref{3.1} for details) and then prove the following result.

\begin{thm}\label{1.2}
 Let $A, T, B$ and $\mathcal{C}$ be as above and let $T$ be self-orthogonal. Then the depth of $\mathcal{C}$ relative to $T$ $\leq$ the finitistic dimension of $B^{op}$.
\end{thm}

For a module $M\in {\rm mod}A$, we use ${\rm Sub}M$ to denote the subcategory cogenerated by $M$. We give the following bijection between bricks in $\mathcal{C}$ and simple modules in $\text{Sub}\mathbb{D}T$, which is partially shown in \cite{As}.

\begin{thm}\label{1.3}
 Let $A, T, S, S_i$ and $B$ be as above. Then the functor
 $$\textup{Hom}_A(T, -): \textup{Fac}T \to \textup{Sub}\mathbb{D}T$$
  induces a bijection between $\{S_1, \dots, S_m\}$ and $\{\text{simple B-modules in \textup{Sub}}\mathbb{D}T\}$.
\end{thm}

By using Theorem \ref{1.3}, we are able to define the $k$-delooping level of $\mathcal{C}$ (See Definition \ref{4.1} for details) and give the following theorem.

\begin{thm}\label{1.4}
   Let $A, T, S, S_i$ and $ B$ be as above. Let $k$ be a positive integer. Then
\begin{itemize}
    \item
    [(1)] The $k$-delooping level of $B$ $\leq$ the $k$-delooping level of $\mathcal{C}$ relative to $T$.
    \item
    [(2)] The finitistic dimension of $B^{op}$ $\leq$ the delooping level of $\mathcal{C}$ relative to $T$.
\end{itemize}
\end{thm}

As a corollary of Theorem \ref{1.2} and Theorem \ref{1.4}, we get that a generalized version of Theorem \ref{1.1} via self-orthogonal support $\tau$-tilting modules (resp. tilting modules).

It is natural to ask whether our Theorem \ref{1.4} can be used for the study of the finitistic dimension conjecture. That is, which class of support $\tau$-tilting modules $T$ with $B={\rm End}_A T$ admits the property that the finitistic dimension of $B^{op}$ is finite?

 We give a positive answer to the question above and prove the following result.

\begin{thm}\label{1.5} Let $A$ be an algebra, $T\in {\rm mod}A$ a support $\tau$-tilting module and $B={\rm End}_A T$. If ${\rm Fac}T\cap {\rm Sub}T$ is of finite representation type, then the finitistic dimension of $B^{op}$ is finite.
  \end{thm}

As an immediate consequence of Theorem \ref{1.5}, we can get the following corollaries.

\begin{cor}\label{1.6} Let $A$ be an algebra of finite representation type and $T\in {\rm mod}A$ a support $\tau$-tilting module and $B={\rm End}_A T$. Then the finitistic dimension of $B^{op}$ is finite.
\end{cor}

\begin{cor}\label{1.7} Let $A$ be a minimal representation infinite algebra, $T\in {\rm mod}A$ a support $\tau$-tilting but not a tilting module and $B={\rm End}_A T$. Then ${\rm findim}B^{op}<\infty$.
\end{cor}

The paper is organized as follows: In Section 2, we recall the preliminaries on tilting modules, $\tau$-tilting modules, torsion pairs and bricks. In Section 3, we introduce the depth of modules relative to a $\tau$-tilting module $T$ and prove Theorem \ref{1.2}. In Section 4, we first prove Theorem \ref{1.3} and then introduce the definition of $k$-delooping level of $\mathcal{C}$ relative to a $\tau$-tilting module $T$ and prove Theorem \ref{1.4}. Then we give examples to illustrate Theorem \ref{1.4}. In Section 5, we give applications of Theorem \ref{1.4} to the finitistic dimension conjecture and show Theorem \ref{1.5}, Corollary \ref{1.6}  and Corollary \ref{1.7}.

Throughout the paper, we assume that all algebras are finite-dimensional basic algebras over an algebraically closed field $K$ and all modules are finitely generated right modules. Denote by $\mathbb{D}$ the ordinary duality and $\tau$ the Auslander-Reiten translation functor.

\section{Preliminaries}

 In this section we recall some basic facts on tilting modules, $\tau$-tilting modules and bricks for later use.

We first recall the definition of torsion pairs from \cite{ASS}.
\begin{defn}\label{2.1}Let $A$ be an algebra. A pair $(\mathcal{T}, \mathcal{F})$ of full subcategories of $\textup{mod}A$ is called a torsion pair if the following conditions are satisfied:
\begin{itemize}

  \item [(1)]  $\textup{Hom}_A(M, N)=0$ for all $M \in \mathcal{T}$, $N \in \mathcal{F}$.

 \item [(2)] $\textup{Hom}_A(M, -)|_\mathcal{F}=0$ implies $M \in \mathcal{T}$.

 \item [(3)] $\textup{Hom}_A(-, N)|_\mathcal{T}=0$ implies $N \in \mathcal{F}$.
\end{itemize}
\end{defn}

For a module $M$, denote by $|M|$ the number of indecomposable direct summands up to isomorphisms. Denote by ${\rm pd}_A M$ the projective dimension of $M$. We recall the definition of tilting modules from \cite{HR}.

\begin{defn}\label{2.2} Let $A$ be an algebra. A module $T\in{\rm mod} A$ is called a tilting module if the following three conditions are satisfied:
\begin{itemize}
  \item[(1)]${\rm pd}_A T\leq1$.
 \item[(2)] $ \text{Ext}^1_A(T,T)=0$.
 \item[(3)] $|T|=|A|$.
\end{itemize}
\end{defn}

We also need the definition of a tilting torsion pair from \cite{ASS}.

\begin{prop}\label{2.3}
Let $A$ be an algebra and $T\in\textup{mod}A$. Then any tilting module $T$ induces a torsion pair $(\mathcal{X}(T_A), \mathcal{Y}(T_A))$ in the category \textup{mod}B, where $B = \textup{End}_A T$ and
$$\mathcal{X}(T_A)=\{X_B|\textup{Hom}_B(X,\mathbb{D}T)=0\}=\{X_B|X \otimes_B T=0\},$$
$$\mathcal{Y}(T_A)=\{Y_B|\textup{Ext}_B^1(Y,\mathbb{D}T)=0\}=\{Y_B|\textup{Tor}_1^B(Y,T)=0\}.$$
\end{prop}

Now we recall the tilting theorem from \cite{ASS} .
\begin{thm}\label{2.4}
Let $A$ be an algebra, $T\in \textup{mod}A$ a tilting module, $B =\textup{End}_A T$, and $(\mathcal{T}(T_A), \mathcal{F}(T_A))$, $(\mathcal{X}(T_A), \mathcal{Y}(T_A))$ be the induced torsion pairs in \textup{mod}$A$ and \textup{mod}$B$, respectively. Then
\begin{itemize}
\item [(1)]  The functors $\textup{Hom}_A(T,-)$ and $-\otimes_B T$ induce quasi-inverse equivalences between $\mathcal{T}(T_A)$ and $\mathcal{Y}(T_A)$.
 \item [(2)] The functors $\textup{Ext}^1_A(T,-)$ and $\textup{Tor}^B_1(-,T)$ induce quasi-inverse equivalences between $\mathcal{F}(T_A)$ and $\mathcal{X}(T_A)$.
\end{itemize}
\end{thm}

We also need the definition of $\tau$-tilting modules from \cite{AIR}.

\begin{defn}\label{2.5} Let $A$ be an algebra, $M\in{\text{mod}A}$ and $P\in\text{mod}A$ be projective.
\begin{itemize}
  \item [(1)]  We call $M$ $\tau$-rigid if $\text{Hom}_A(M,\tau M)=0$.

 \item [(2)]  We call $M$ $\tau$-tilting if $M$ is $\tau$-rigid and $|M|= |A|$.

 \item [(3)]  We call $M$ support $\tau$-tilting if there exists an idempotent $e$ of $A$ such that $M$ is a $\tau$-tilting $(A/\langle e\rangle)$- module.

\item [(4)] $(M, P)$ is a $\tau$-rigid pair if $M$ is $\tau$-rigid and $\text{Hom}_A(P, M)=0$.

\item [(5)]  $(M, P)$ is a  support $\tau$-tilting pair (resp. an almost support $\tau$-tilting pair) if $(M,P)$ is a $\tau$-rigid pair and $|M|+|P|=|A|$ (resp. $|M|+|P|=|A|-1$ ).

 \end{itemize}
\end{defn}

The following result on the relations of $\tau$-tilting modules and tilting modules in \cite[Proposition 2.2]{AIR} is essential in this paper.

\begin{prop}\label{2.12} Let $A$ be an algebra, $M\in{\textup{mod}A}$ with right annihilator $I=\textup{ann} M$ and $\bar{A}=A/I$.
If $M$ is a support $\tau$-tilting module in $\textup{mod}A$, then it is a tilting module in $\textup{mod} \bar{A}$.
\end{prop}

We also need the following result which generalizes \cite[Proposition 2.3]{Z}.

\begin{prop}\label{2.6} Let $A$ be an algebra and let $T\in\textup{mod}A$ be a support $\tau$-tilting module. For any $M\in\textup{Fac}T$, there is an exact sequence $\cdots\rightarrow T_1\stackrel{f_1}{\rightarrow}T_0\stackrel{f_0}{\rightarrow}M\rightarrow 0$ with $T_i\in \textup{add}T$ and $\textup{Ker} f_i\in \textup{Fac}T$.
\end{prop}

\begin{proof} By Proposition \ref{2.12}, one gets that $T$ is a tilting module in ${\rm mod}\bar{A}$. Take a minimal ${\rm add}T$-approximation of $M \in {\rm Fac}T\subseteq {\rm mod}\bar{A}$, one gets the following exact sequence $$0\rightarrow {\rm Ker}f_0\rightarrow T_0\rightarrow M\rightarrow 0.$$  Then one gets that ${\rm Ker}f_0$ satisfies ${\rm Ext}_{\bar{A}}^1(T,{\rm Ker}f_0)=0$, which implies that ${\rm Ker}f_0\in {\rm Fac}T$ since $T$ is a tilting module. Continue the same process, the assertion holds.
\end{proof}
By using Proposition \ref{2.6}, for any $M\in\mathcal{C}={\rm Fac} T$, we use $\Omega^i_T M$ the $(i+1)$-th syzygy of $M$ relative to $T$. Denote by $\underline{\mathcal{C}}_T$ the stable category of $\mathcal{C}$ modulo ${\rm add}T$. Then $\Omega^1_T: \underline{\mathcal{C}}_T\longrightarrow \underline{\mathcal{C}}_T$ is an additive functor.

We also need the following definition of mutations for support $\tau$-tilting modules from \cite[Definition 2.19]{AIR}.

\begin{defn}\label{2.10} Let $A$ be an algebra and let $(M,P)$ and $(N,Q)$ be support $\tau$-tilting pairs in $\text{mod}A$. If there exists an almost support $\tau$-tilting pair that is a common direct summand of $(M,P)$ and $(N,Q)$, then we say that$(N,Q)$ is a mutation of $(M,P)$.  For each direct indecomposable summand $L$ of $M$ or $P$, there uniquely exists a mutation of $(M,P)$ at $L$.  If $\text{Fac}M \subseteq \text{Fac}N$, then $N$ is called a left mutation of $M$. Dually, $N$ is called a right mutation of $M$.
\end{defn}

 Now we recall the following definitions of bricks and semi-bricks, see \cite{As, Ri} for details.

\begin{defn}\label{2.7}  Let $A$ be an algebra and $S\in {\rm mod}A$.
\begin{itemize}
  \item
[(1)]  $S$ is called a brick if $\text{End}_A (S)$ is a division $K$-algebra.

  \item
[(2)] $S$ is called a semi-brick if $S\simeq \oplus_{i=1}^m S_i$ and $\text{Hom}_A (S_i, S_j) =0$ holds for any $i\neq j$, where $S_i$ is a brick and $m$ is a positive integer.
\end{itemize}
\end{defn}

 For a full  subcategory  $\mathcal{C}\subseteq \text{mod}A$. Denote by
$\text{add}\mathcal{C}$ the additive closure of $\mathcal{C}$,
 $\text{Fac}\mathcal{C}$ the subcategory of the factor modules of objects in $\text{add}\mathcal{C}$,
 and $\text{Filt}(\mathcal{C})$ the subcategory of the objects $M$ such that there exists a sequence
$0= M_0 \subseteq M_1 \subseteq M_2 \cdots  \subseteq M_n = M$ with  $M_i/M _{i-1} \in$ ~add$\mathcal{C}$.
We use $T(\mathcal{C})$ to denote the subcategory $\text{Filt}(\text{Fac}\mathcal{C})$.

The following property on semi-bricks \cite[Lemma 2.7(1)]{As} is essential in this paper.

\begin{prop}\label{2.8}
Let $A$ be an algebra and $S \in \textup{mod}A$ a semi-brick. Let $S_1$ be a brick which is a direct summand of $S$. Every nonzero homomorphism $f: M \to S_1$ with $M \in T(S)$ is surjective. Moreover, we have $\textup{Ker}f \in T(S)$.
\end{prop}

We also need the following proposition on relations between support $\tau$-tilting modules and semi-bricks, see \cite[Proposition 2.9(1), Theorem 2.3(2)]{As}.

\begin{thm}\label{2.9} Let $A$ be an algebra and $T\in\textup{mod}A$.
If $T$ is a support $\tau$-tilting module, then $ \textup{Fac}T = T(S)=\textup{Filt(Fac}S)$, where $S$ is a the left finite semi-brick according to $T$ in \cite[Theorem 2.3(2)]{As}.
\end{thm}

Now we recall the definition of label support $\tau$-tilting quivers with bricks from \cite[Definition 2.14]{As}.

\begin{defn}\label{2.11} Let $A$ be an algebra and let $M\in\textup{mod}A$ be a support $\tau$-tilting module and decompose $M$ as $M =\oplus_{i=1}^t M_i$
with $M_i$ indecomposable. Assume that $M \to N$ is an arrow in the support $\tau$-tilting quiver of $A$, and that
$N$ is the left mutation of $M$ at $M_i$. Then we label this arrow with a brick $S_i=M_i/\Sigma_{f\in \text{rad}\text{Hom}_A(M, M_i)}{\text{Im}f}$.
\end{defn}

\section{Depth relative to $T$}

In this section, we assume that $A$ is an algebra and $T\in \text{mod} A$ is a support $\tau$-tilting module. Denote by $\text{Fac}T=T(S)=\text{Filt(Fac} S)$, where $S=S_1\bigoplus \cdots \bigoplus S_m$ is the left finite semi-brick in \cite[Theorem 2.3(2)]{As}
 and $S_i$ is a brick. Then we show the main result on the depth relative to $T$ which generalizes \cite[Proposition 1.3]{G}.


Let $A$ be an algebra. Recall from \cite{AuB} that the grade of $M\in \text{mod}A$ denoted by $\text{grade}M$ is defined as the infimum of the integer $i\geq 0$ such that $\text{Ext}^i_A(M,A)\neq 0$. Recall from \cite{G} that the depth of $A$ denoted by $\text{depth}A$ is defined as the supremum of the grade of simple modules.

In the following we give the definition of grade and depth relative to $T$.

\begin{defn}\label{3.1} Let $A$ and $T$ be as above. Denote by $\mathcal{C}=\text{Fac}T$.
\begin{enumerate}
\item The grade of $M\in \mathcal{C}$ is defined as
$\text{grade}_T M =\text{inf}\{i\geq 0 |\text{Ext}^i_A(M,T)\neq 0\}.$
\item The depth of $\mathcal{C}$ is defined as $\text{depth}_T \mathcal{C}={\sup\limits_{1\leq i\leq m}} \text{grade}_T S_i$.
\end{enumerate}
\end{defn}
 It is easy to see that the grade of $M$ relative to $T$ we defined is the classical grade whenever $T=A$.
We give the following example to illustrate Definition \ref{3.1}.

\begin{exm}\label{3.2} Let $A$ be an algebra given by the quiver $\xymatrix{1\ar[r]^a&2\ar[r]^b&3}$ with the relation $ab=0$. Then
\begin{itemize}
    \item
    [(1)] $T=\substack{ 1\\2\\ }\oplus1\oplus3$ is a $\tau$-tilting module but not a tilting module and $\mathcal{C}=\textup{Fac} T=\textup{add}(1\oplus3\oplus\substack{1\\2})$.
    \item
    [(2)] Every indecomposable module in $\textup{Fac}T$ has $T$-grade $0$.
    \item
    [(3)] The semi-brick relative to $T$ is $S=3\oplus\substack{1\\2}$. So the $\textup{depth}_T \mathcal{C}=0$.
\end{itemize}
\end{exm}

To prove the main result we need the following lemma from \cite[Page 8]{Xi}.

\begin{lem}\label{3.a} Let $A$ be an algebra and $T\in \textup{mod}A$. Denote by $B=\textup{End}_A T$. If $T_1\stackrel{f_1}{\rightarrow} T_0\stackrel{f_0}{\rightarrow}M\rightarrow 0$ is a minimal ${\rm add}T$-presentation of $M$, then $\textup{Hom}_A(T_0,T)\stackrel{{f_1}^T}{\rightarrow}\textup{Hom}_A(T_1,T)\rightarrow \textup{Coker} {f_1}^T\rightarrow 0$ is a minimal projective presentation of $\textup{Coker} {f_1}^T$ in $\textup{mod}B^{op}$.
\end{lem}

Let $A$ be an algebra. Recall from \cite{AuB}, the finitistic dimension of $A$ is defined as:
$$
\text{findim}A= \sup\left\{\text{pd}_A M \mid M \in \text{mod}A,\ \text{pd}_A M < \infty \right\}.
$$

Dually, we use findim$A^{op}$ to denote the finitistic dimension of $A^{op}$.

Recall that a module $M$ is called self-orthogonal if ${\rm Ext}_A^i(M,M)=0$ for all $i\geq 1$. Now we are in a position to show our main result in this section.

\begin{thm}\label{3.3}
 Let $A$ be an algebra and $T\in \textup{mod}A$ be a self-orthogonal support $\tau$-tilting module (resp. a tilting module) with
$\textup{Fac}T= \textup{Filt(Fac}S)$, where $S=S_1 \bigoplus\cdots \bigoplus S_m$ is a semi-brick and $S_i$ is a brick. Denote by $B = \textup{End}_AT$. Then the depth of $\mathcal{C}$ relative to $T$ $\leq$ the finitistic dimension of $B^{op}$, that is, ${\rm depth}_T \mathcal{C}\leq {\rm findim} B^{op}$.
\end{thm}

\begin{proof}
 Let $$T_n \stackrel{f_n}{\to} \cdots \to T_1\stackrel{f_1} {\to} T_0 \to S_i \to 0\ (3.1)$$ be a minimal $T$-presentation of the brick $S_i$. We divide the proof into two cases.

\begin{itemize}
  \item [(1)] If there exists a brick $S_i \in \text{Fac}T$ such that $\text{Ext}^j_A (S_i, T) =0$ for all $j\geq 0$, then  findim$ B^{op}=+\infty$.

Applying the functor ${\rm Hom}_A(-, T)$ to $(3.1)$, since $T$ is self-orthogonal, we get the  following long exact sequence
$$0\to \text{Hom}_A(T_0, T)\to \cdots \xrightarrow{f_n^T} \text{Hom}_A(T_n, T)\to \text{coker} f_n ^T \to 0 \ (3.2)$$

\noindent Then by Lemma \ref{3.a} one gets that $(3.2)$ is a minimal projective resolution of $\text{coker} f_n ^T$ and hence $\text{Ext} _B^n (\text{coker}f_n^T, B) \neq 0$. Therefore the projective dimension of coker$f_n^T =n$ and therefore findim$B^{op}= +\infty.$

\item [(2)]  If $\prod _{j\geq 0}\text{Ext} _A^{j} (S_i, T) \neq 0$ holds for each brick $S_i$, then there exists minimal non-negative integer $t$ such that $\text{Ext}^t_A(S_i, T) \neq 0$. We show that three is a module $C\in {\rm mod}B^{op}$ of projective dimension $t$.

If $t=0$, then we have $\text{Ext}^0_A(S_i, T) \neq 0$, therefore the projective dimension of $\text{Hom}_A(S_i, T) \geq 0$.

In the following, we show the assertion holds for $t\geq1$. Applying the functor $\text{Hom}_A(-, T)$ to  $(3.1)$, since $T$ is self-orthogonal and $t$ is minimal, we get the following long exact sequence
$$0\to \text{Hom}_A(T_0, T)\to\cdots \stackrel{{f_t}^T}{\to} \text{Hom}_A(T_t, T)\to \textup{Coker}{f_t}^T \to 0\ (3.3)$$
 which is a minimal projective resolution of $\textup{Coker}{f_t}^T$ by Lemma \ref{3.a}. Moreover, $\text{Ext}^t_A(S_i, T)$ is a submodule of $\textup{Coker}{f_t}^T$. 
 
 In all, we have shown that depth$_T \mathcal{C} \leq \text{findim}B^{op}$.\qedhere
\end{itemize}\noindent
\end{proof}

We should remark that a self-orthogonal support $\tau$-tilting module need not be a tilting module, see \cite[Example 3.6]{Z}. Taking $T=A$ for a finite-dimensional algebra $A$, we have the following corollary \cite[Proposition~1.3]{G}.

\begin{cor}\label{3.4}
Let $A$ be an algebra. We have the inequality
$${\rm depth}A \leq {\rm findim}A^{op}.$$
\end{cor}

We end this section with the following example.

\begin{exm}\label{3.5} Let $A$ be an algebra given by the quiver \[\xymatrix{
1\ar[r]^{a}&2\ar[r]^{b}&3
}\]  with the relations $ab=0$. Then
\begin{itemize}
    \item
    [(1)] $T=\substack{1\\2}\oplus\substack{2\\3}\oplus \substack{2}$ is a tilting module.
    \item
    [(2)] $B=\textup{End}_A T$ is given by \[\xymatrix{
1\ar[r]^{\alpha}&2\ar[r]^{\beta}&3} \]and $\textup{gldim}B=1$.
    \item
    [(3)] $\textup{Fac}T=\textup{FiltFac}(\substack{1}\oplus \substack{2\\3})$, one gets that
    $\textup{depth}_T \mathcal{C}=1\leq \textup{findim}B^{op}=1$.
\end{itemize}
\end{exm}

\section{Delooping level relative to $T$}

 Throughout this section, $T=\oplus_{i=1}^nT_i$ is a support $\tau$-tilting module with $T_i$ indecomposable and $\text{Fac}T=T(S)=\text{Filt(Fac}S)$, where $S\simeq S_1\bigoplus \cdots \bigoplus S_m$ is the left finite semi-brick in \cite[ Theorem 2.3(2)]{As}
 and $S_i$ is a brick. Denote by $B=\text{End}_A T$. Following \cite{CLZZ}, in this section we define the delooping level relative to $T$ for the subcategory $\mathcal{C}=\text{Fac}T$ and then study the relation between the finitistic dimension of $B^{op}$ and the delooping level of $\mathcal{C}$.

To define the delooping level relative to $T$, we give the following result which gives a bijection between bricks $S_i$ and simple $B$-modules in $\text{Sub}\mathbb{D}T$. We should remark that part of this result is first shown in the proof of \cite[Theorem 2.27]{As}.

\begin{thm}\label{4.2}
 Let $A, T, S, S_i$ and $ B$ be as above. Then the functor
 $$\textup{Hom}_A(T, -): \textup{Fac}T \to \textup{Sub}\mathbb{D}T$$
  induces a bijection between $\{S_1, \dots, S_m\}$ and $\{\text{simple $B$-modules in \textup{Sub}}\mathbb{D}T\}$.
\end{thm}

\begin{proof} We divide the proof into two steps.

 (1) We show that $\textup{Hom}_A(T, S_i)$ is simple. This is actually shown by Asai in \cite{As}. Here we give a new proof.

On the contrary, suppose that $\textup{Hom}_A(T, S_i)$ is not simple. Since $S_i$ is a brick and a direct summand of the left finite semi-brick $S$, one gets an exact sequence
$$ 0 \to L_i \to T_i \to S_i \to 0 \ \ \ \ (4.1)$$
  with $T_i$ an indecomposable direct summand of $T$ and $L_i \in \textup{Fac}T$ by Proposition \ref{2.8}. Applying $\textup{Hom}_A(T, -)$ to $(4.1)$, one gets an exact sequence
$$ 0 \to \textup{Hom}_A(T, L_i) \to \textup{Hom}_A(T, T_i) \to \textup{Hom}_A(T, S_i) \to 0.$$
On the other hand, one has the following exact sequence
$$0 \to R  \to \textup{Hom}_A(T, T_i) \to S \to 0$$
with $S$ simple in $\text{mod} B$.
We have the following commutative diagram

\xymatrix{
0  \ar[r] &\textup{Hom}_A(T, L_i) \ar[r]\ar[d]^{\theta}& \textup{Hom}_A(T, T_i)\ar[r]\ar@{=}[d]
& \textup{Hom}_A(T, S_i) \ar[r]\ar[d]^{\pi} & 0 \\
 0 \ar[r] & R \ar[r] &\textup{Hom}_A(T, T_i) \ar[r]& S \ar[r]& 0.
}

By using the snake lemma, one gets that $\theta$ is a monomorphism and $\pi$ is an epimorphism with $\textup{Coker}\theta \cong \textup{Ker} \pi = N$. So the exact sequence $$0 \to N \stackrel{a}{\to} \textup{Hom}_A(T, S_i) \to S \to 0$$ implies that $N \in \textup{Sub}\mathbb{D}T$. Since $\textup{Hom}_A(T, -): \textup{Fac}\ T \to \textup{Sub}\mathbb{D}T$ is an equivalence, there exists $ M \in \textup{Fac}T$ such that $N \cong \textup{Hom}_A(T, M)$.
Then one gets an exact sequence
$$0 \to \textup{Hom}_A(T,M) \stackrel{a}{\to} \textup{Hom}_A(T, S_i) \to S \to 0. \ \ \ (4.2)$$
So there exists $ b: M \to S_i \ne 0$ such that $a= \textup{Hom}_A(T, b)$.

By Proposition \ref{2.8}, $b$ is surjective and $Q=\textup{Ker} b\in \textup{Fac}T $, then one gets an exact sequence
$$ 0 \to Q  \to M \stackrel{b}{\to} S_i \to 0.\ \ \ \ \ \ (4.3)  $$
Applying $\textup{Hom}_A(T, -)$ to $(4.3)$, we get the following exact sequence
$$0 \to \textup{Hom}_A(T, Q) \to \textup{Hom}_A(T, M) \xrightarrow{a} \textup{Hom}_A(T, S_i) \to 0.\ \ \ \ \ (4.4)$$
Comparing $(4.2)$ and $(4.4)$, one gets that the map $a$ is an isomorphism which implies that $\textup{Hom}_A(T, Q) = 0$, and hence $Q=0$. Then $M \cong S_i$ implies $S=0$ by using $(4.2)$, contradiction.

(2) We show for any $S'$ simple in $\textup{Sub}\mathbb{D}T$, there is $S_i \in \{S_1, \dots, S_m\}$ such that $\textup{Hom}_A(T, S_i) \cong S'$.

Since $\textup{Hom}_A(T, -): \textup{Fac}T \to \textup{Sub}\mathbb{D}T$ is an equivalence by Proposition \ref{2.12} and Theorem \ref{2.4}, then there exists $ M \in \textup{Fac}T$ such that $\textup{Hom}_A(T, M) \cong S'$.
So the fact $$ \textup{Hom}_A(M, M) \cong \textup{Hom}_B(\textup{Hom}_A(T, M), \textup{Hom}_A(T, M)) \cong \textup{Hom}_B(S', S')
\cong K$$ implies that $M$ is a brick.

Since $M \in \textup{Fac}T = \textup{Filt} (\textup{Fac}S)$, then there exists a brick $S_i$ which is a direct summand of $S$ such that $S_i \xrightarrow{f} M \ne 0$.
Since $\textup{Hom}_A(T, -): \textup{Fac}T \to \textup{Sub}\mathbb{D}T$ is an equivalence,
one gets
\begin{align*}
\textup{Hom}_A(T, -): &\textup{Hom}_A(S_i, M) \to \textup{Hom}_B (\textup{Hom}_A(T, S_i), \textup{Hom}_A(T, M))\\
\ \ \ \ \ \ &\ \ \ \ \ \ \ \ \ \ \ \ \ \ \ \  f \to \textup{Hom}_A(T, f)
\end{align*}
induces an isomorphism.
So the map $\textup{Hom}_A(T, f): \textup{Hom}_A(T, S_i){\rightarrow} \textup{Hom}_A(T, M)$ is not zero, which implies that $\textup{Hom}_A(T, f)$ is an isomorphism since both $\textup{Hom}_A(T, S_i)$ and $\textup{Hom}_A(T, M)$ are simple. Using $\textup{Hom}_A(T, -): \textup{Fac}T \to \textup{Sub}\mathbb{D}T$ is an equivalence again, one gets that $S_i\cong M$.
\end{proof}

As a consequence, we have the following corollary.

\begin{cor}\label{4.3} Let $A, S, T$ and $B$ be as above and $|T|=|A|=n$. Then every simple $B$-module is in $\textup{Sub}\mathbb{D}T$ if and only if $T\cong A$.
\end{cor}
\begin{proof} If $T\cong A$, there is nothing to show. Conversely, by Theorem \ref{4.2}, the left finite semi-brick $S$ should have $n$ non-isomorphic indecomposable direct summands and $S=S_1\oplus S_2\oplus\dots S_n$. Since $\textup{Sub}\mathbb{D}T$  is closed under extensions and every simple $B$-module is in $\textup{Sub}\mathbb{D}T$, then one gets that $\textup{Sub}\mathbb{D}T={\rm mod}B$. By Theorem \ref{2.4}, ${\rm mod} B\simeq {\rm Fac}T$ as exact categories. One gets that ${\rm Fac}T={\rm Filt} S$ is a wide subcategory by \cite[Proposition 2.17]{As}. Since ${\rm Fac}T$ is a torsion class, one gets that it is a Serre subcategory, that is, closed under submodules, factor modules and extensions. Since $T$ is sincere, then every simple $A$-module is a composition factor of $T$, and hence in ${\rm Fac}T$. Then ${\rm Fac}T={\rm mod}A={\rm Fac}A$, and hence $T\cong A$ holds.
\end{proof}

Now we show some basic homological properties between $\text{Fac}T$ and $\text{mod}B$ for later use.

\begin{prop}\label{4.4}
Let $A, T, S, S_i$ and $ B$ be as above.
\begin{itemize}
\item[(1)] If $\cdots\to T_1\to T_0\to N\to 0 $ is a minimal $T$-resolution of $N$ in $\mathcal{C}$, then $\cdots\to \textup{Hom}_A(T,T_1)\to \textup{Hom}_A(T, T_0)\to\textup{Hom}_A(T, N)\to 0$ is a minimal projective resolution of $\textup{Hom}_A(T, N)$.
\item [(2)] For any $n>d\geq1$ and $M\in\textup{mod}B^{op}$, we have
$\textup{Tor}^B_n (\textup{Hom}_A(T, S_i), M)\cong \textup{Tor}^B_{n-d} (\textup{Hom}_A (T, \Omega^d_T S_i), M).$
\end{itemize}
\end{prop}

\begin{proof}(1) By Proposition \ref{2.6}, one gets that every kernel in a minimal $T$-resolution is in $\mathcal{C}$. Applying the functor ${\rm Hom}_A(T, -)$  to the minimal $T$ presentation $$\cdots\to T_1\to T_0\to N\to 0 ,$$ one gets the following long exact sequence $$\cdots\to \textup{Hom}_A(T,T_1)\to \textup{Hom}_A(T, T_0)\to\textup{Hom}_A(T, N)\to 0.$$ It is a minimal projective resolution since $\text{Hom}_A(T,-)$ is an equivalence.

(2) Denote $L=\text{Hom}_A(T, S_i)$.
Denote by $T_* \xrightarrow{\simeq} S_i$ a $T$-resolution of a brick $S_i$ by Proposition \ref{2.6}. Then one has the following exact sequence
$$0\to \Omega_T^d S_i \to  T_{d-1} \to \cdots \to T_1 \to T_0 \to S_i \to 0\ \ \ \ \ (4.5)$$ with $\Omega_T^d S_i\in\text{Fac}T$.

Applying  $\text{Hom}_A (T, -)$ to  $(*5)$, we have the following long exact sequence
 $$0\to \text{Hom}_A(T,\Omega_T^d S_i) \to  \text{Hom}_A(T,T_{d-1} )\to \cdots \to \text{Hom}_A(T, T_0) \to \text{Hom}_A(T, S_i) \to 0\ \ \ \ (4.6),$$
which is a projective resolution of $L=\text{Hom}_A(T, S_i)$.

Applying the functor  $~_-\otimes_B M$ to  $(4.6)$, one gets the following long exact sequence:
\begin{align*}
\cdots\to\textup{Tor}_n^B(P_0, M) \to \textup{Tor}_n^B(L, M)\to \textup{Tor}_{n-1}^B(\Omega^1L, M)\to \textup{Tor}_{n-1}^B(P_0, M)\to\cdots.
\end{align*}
By using dimension shifting, we have
$$\text{Tor}_n^B (L, M) \cong \text{Tor}_{n-1}^B (\Omega^1 L, M) \cong \text{Tor}_{n-2}^B (\Omega^2 L, M) \cong \text{Tor}_{n-d}^B (\Omega^d L, M).
$$
We are done.
\end{proof}

  Let $A$ be an algebra, $M\in \text{mod}A$ and let $ k \in \mathbb{Z^+}$. Recall from \cite{GuI} that the $k$-delooping level of $M$ denoted by $k$-$\text{dell}M$ is defined as the infimum of $d\geq 0$ such that $\Omega^d M$ is a stable retract of $\Omega^{d+k}N$ for $N$ in $\text{\underline{mod}}A$. The $k$-delooping level of $A$ denoted by $k$-\text{dell}$A$ is defined as the supremum of the $k$-delooping level of simple $A$-modules. Putting $k=1$, one gets the delooping level of $M$ and $A$ introduced in \cite{G}.

For a module $M$ in $\text{mod}A$, denote by $\text{rad} M$ the radical of $M$. By Proposition \ref{2.6} and Theorem \ref{4.2}, we are able to give the following definition.

\begin{defn}\label{4.1} Let $A$, $T$ and $S_i$ be as above. Let $k$ be a positive integer.
\begin{itemize}
\item [(1)] For any $ M\in \mathcal{C}=\text{Fac}T$, the $k$-delooping level of $M$ relative to $T$ denoted by $k$-${\rm dell}_T M$ is defined as the infimum of integer $d$ such that $\Omega_{T}^{d} M$ is a stable retract of $\Omega_T^{d+k}N$ for some $N\in \text{Fac}T$.
\item [(2)] The $k$-delooping level of $\mathcal{C}$ relative to $T$ is defined as the supremum of the $k$-delooping level relative to $T$ of $S_i$ for $1\leq i\leq m$ and the $(k+1)$-delooping level of $N_j$ $+1$ for $m+1\leq j\leq n$, where $\text{Hom}_A(T,N_j)\cong \text{rad}\text{Hom}_A(T,T_j)$ .
\item [(3)] The global $k$-delooping level of $\mathcal{C}$ relative to $T$ is defined as the  supremum of the $k$-delooping level relative to $T$ of $M$ for all $M\in\mathcal{C}$.
\end{itemize}
\end{defn}
 If $T$=$A$, then one gets that the semi-brick $S$ should be the direct sum of simple $A$-modules. Therefore the $k$-delooping level of $\mathcal{C}$ is the classical case introduced by \cite{GuI}. It is easy to see that the $k$-delooping level of $\mathcal{C}$ relative to $T$ is no more than the global $k$-delooping level of $\mathcal{C}$ relative to $T$. From now on, we use $k$-$\text{dell}_TM$, $k$-$\text{dell}_T\mathcal{C}$ and global-$k$-$\text{dell}_T\mathcal{C}$ to denote the three notations above.

Now we are in a position to show the main result in this section.

\begin{thm}\label{4.5}
   Let $A, T, S, S_i$ and $ B$ be as above. Let $k$ be a positive integer. Then
\begin{itemize}
    \item
    [(1)] The $k$-delooping level of $B$ is no more than the $k$-delooping level of $\mathcal{C}$ relative to $T$, that is, $k$-${\rm dell}B\leq k$-${\rm dell}_T\mathcal{C}$.
    \item
    [(2)] The finitistic dimension of $B^{op}$ is no more than the delooping level of $\mathcal{C}$ relative to $T$, that is, ${\rm findim}B^{op}\leq {\rm dell}_T \mathcal{C}$.
\end{itemize}
\end{thm}

\begin{proof} (1) We divide the proof into three steps.

(a) For any $X\in \mathcal{C}$, we show that  $k$-${\rm dell}_B {\rm Hom}(T, X)\leq $ $k$-${\rm dell}_T X$ for all $k\geq 1$.

If $k$-${\rm dell}_T X=\infty$, then there is nothing to show. Now assume that $k$-${\rm dell}_T X=t<\infty$. Then there is a module $Y\in\mathcal{C}$ such that $\Omega^{t}_T X$ is a stable retract of $\Omega^{k+t}_T Y$ in $\underline{\mathcal{C}}_T$. Then by Theorem \ref{2.4} and Proposition \ref{4.4}(1), one gets that $\Omega^{t}_B {\rm Hom}_A(T, X)$ is a stable retract of $\Omega^{k+t}_B {\rm Hom}_A(T, Y)$ in $\underline{\rm{mod}}B$. The assertion holds.

(b) For any simple $D_i \in{\rm mod} B$ with $1\leq i\leq m$, if $ D_i\cong {\rm Hom}(T, S_i)$ by Theorem \ref{4.2}, then one gets that
the $k$-delooping level of $\text{Hom}_A(T, S_i)$ is no more than that of $S_i$ for $1\leq i\leq m$ by (a).

(c)  For any simple $D_j \not\cong {\rm Hom}_A(T, S_j)\in{\rm mod} B$ with $m+1\leq j\leq n$, we show $k$-${\rm dell}_B D_j\leq (k+1)$-${\rm dell}_T N_j +1$, where $\text{Hom}_A(T, N_j)\cong\text{rad}\text{Hom}_A(T,T_j)$.

Since $\text{Hom}_A(T,T_j)$ is an indecomposable projective $B$-module, one has the following exact sequence:
$$0\rightarrow \text{rad}\text{Hom}_A(T,T_j)\to \text{Hom}_A(T,T_j)\to D_j\to 0 $$
By Theorem \ref{2.4}, $\text{Hom}_A(T,-): \text{Fac}T\to\text{Sub}\mathbb{D}T$ is an equivalence and $\text{Sub}\mathbb{D}T$ is a torsionfree class,
and hence $\text{rad}\text{Hom}_A(T,T_j)$ is in $\text{Sub}\mathbb{D}T$. So there is a module $N_j\in\mathcal{C}$ such that $\text{Hom}_A(T, N_j)\cong\text{rad}\text{Hom}_A(T,T_j)$.

If $(k+1)$-${\rm dell}_T N_j=d<\infty$, then one gets that $\Omega_T^{d}N_j$ is a stable retract of $\Omega_T^{d+k+1}L_j$ in $\underline{\mathcal{C}}_T$ for some $L_j\in \mathcal{C}$. Applying the functor ${\rm Hom}_A(T,-)$, one has that $\Omega_B^{d}{\rm Hom}_A(T, N_j)$ is a stable retract of $\Omega_B^{d+k+1}U$ in $\underline{\rm mod}B$ with $U\in{\rm Sub}\mathbb{D}T$, therefore $\Omega_B^{d+1}D_j$ is a stable retract of $\Omega_B^{d+k+1}U$ in $\underline{\rm mod}B$. So $k$-${\rm dell}_B D_j\leq d+1$. Using a similar argument, the inequality holds for $d=\infty$.
We are done.

(2) Let $k=1$. It is a straight result of (1) and \cite[Proposition 1.3,]{G}.
\end{proof}

We have the following straight corollary which is a $\tau$-tilting version of Theorem \ref{1.1}.

\begin{cor}\label{4.6}
 Let $A$ be a finite-dimensional algebra, $T\in{\rm mod}A$ a self-orthogonal support $\tau$-tilting module (resp. a tilting module) with $\mathcal{C}={\rm Fac}T$ and $B={\rm End}_A T$. Then
 ${\rm depth}_T \mathcal{C}\leq {\rm findim} B^{op}\leq {\rm dell}_T\mathcal{C}$.
\end{cor}

\begin{proof} This is the straight result of Theorem \ref{3.3} and Theorem \ref{4.5}.
\end{proof}

We also have the following corollary.

\begin{cor}\label{4.7}
 Let $A,T, B$ and $\mathcal{C}$ be as in Corollary \ref{4.6}.
\begin{itemize}
    \item
    [(1)] $\textup{depth}_T \mathcal{C} = 0$.
    \item
    [(2)] $\textup{findim}B^{op} = 0$.
    \item
    [(3)] $\textup{dell}_T \mathcal{C} = 0$.
\end{itemize}
Then we have $(3)\Rightarrow(2)\Rightarrow (1)$.
\end{cor}

\begin{proof}
This is an immediate consequence of Corollary \ref{4.6}.
\end{proof}

In Theorem \ref{4.5}, we show that the finitistic dimension of $B^{op}$ is bounded by the delooping level of $\mathcal{C}$. But the delooping level of  $\mathcal{C}$ is not easy to compute by using Definition \ref{4.1}. It is natural to ask when the finitistic dimension of $B^{op}$ can be bounded by the supremum of the delooping level of $S_i$ ?

\begin{thm}\label{4.8} Let $A, T, S, S_i$ and $B$ be as above. Let $m\geq 1$, $S_i\neq 0$ for $1\leq i\leq m$ and $S_j=0$ for $m+1\leq j\leq n$. If the projective dimension of the top of $\textup{Hom}_A(T,T_j)$ is $l_j$ for $m+1\leq j\leq n$, then $\textup{findim}B^{op}\leq \textup{sup}\{l_j|m+1\leq j\leq n\ \textup{and}\ \textup{dell}_T S_i|1\leq i\leq m\}$.
\end{thm}
\begin{proof} Assume that $\textup{sup}\{l_j|m+1\leq j\leq n\ \textup{and}\ \textup{dell}_T S_i|1\leq i\leq m\}=d$, we show $\textup{findim}B^{op} \leq d$.

For any $M \in \textup{mod}B^{op}$ with the projective dimension $t>d$, there exists a simple module $ D\in \textup{mod}B$ such that $\textup{Tor}^B_t (D, M) \neq 0$. By the assumption, $D$ is in $\text{Sub}\mathbb{D}T$. Then by Theorem \ref{4.2}, one gets a brick $S_i\in \mathcal{C}$ such that $\text{Hom}_A(T,S_i)\cong D$. In the following, we show that $t \leq d.$ On the contrary, suppose that $t>d$. Then by Theorem \ref{4.2} and Proposition \ref{4.4}(2), one can get that
\begin{align*}
\text{Tor}_{t}^B (D, M) = \text{Tor}_{t}^B (\textup{Hom}_\Lambda(T,S_i), M) \cong \text{Tor}_{t-d}^B (\Omega^d \textup{Hom}_A(T,S_i), M) \\ \cong \text{Tor}_{t-d}^B (\textup{Hom}_A(T,\Omega_T^d S_i), M),
\end{align*}

 which is a direct summand of

 \begin{align*}
 \text{Tor}_{t-d}^B (\textup{Hom}_A(T,\Omega_T^{d+1} N), M) \cong \text{Tor}_{t-d}^B (\Omega^{d+1}\textup{Hom}_A(T, N), M)\\ \cong \text{Tor}_{t-d+d+1}^B (\textup{Hom}_A(T, N), M)=0
\end{align*} This is a contradiction. The assertion holds.
\end{proof}

Immediately we have the following corollaries.

\begin{cor}\label{4.9} Let $A$ be an algebra and $T=P(1)\oplus P(2)\oplus\cdots\oplus T_j\oplus\cdots\oplus P(n)$ be a $\tau$-tilting module given by a left mutation from $A$. Let $S, S_i$ be as in Definition \ref{4.1}. If the top of $\textup{Hom}_A(T, T_j)$ is of projective dimension $l$, then $\textup{findim}B^{op}\leq \textup{sup}\{l\ \textup{and}\ \textup{dell}_T S_i|1\leq i\leq n, i\neq j\}.$
\end{cor}

\begin{proof} Since $P(i)$ is not in $\text{Fac} T/P(i)$, by Definition \ref{2.10}, one gets that there is a left mutation of $T$ at $P(i)$. Then by Definition \ref{2.11}, one gets that the corresponding brick $S_i$ is non-zero. Then the assertion holds.
\end{proof}

\begin{cor}\label{4.10} Let $A$ be an algebra and $T=P(1)\oplus P(2)\oplus\cdots\oplus T_j\oplus\cdots\oplus P(n)$ be an APR-tilting module. Let $ S_i$ be as in Definition \ref{4.1}. If the top of $\textup{Hom}_A(T, T_j)$ is of projective dimension $l$, then $\textup{findim}B^{op}\leq \textup{sup}\{l\ \textup{and}\ \textup{dell}_T S_i|1\leq i\leq n, i\neq j\}.$
\end{cor}

We give the following example to illustrate Theorem \ref{4.5} and \ref{4.8}, which also shows why Definition \ref{4.1}(2) is reasonable.

\begin{exm}\label{4.10} Let $A$ be an algebra given by the quiver \[\xymatrix{
1\ar@<2pt>[r]^{\alpha}&2\ar@<2pt>[l]^{\beta}
}\]  with the relations $\alpha\beta\alpha=\beta\alpha\beta=0$. Then
\begin{itemize}
    \item [(1)] $T=\substack{1\\2\\1}\oplus1$ is a $\tau$-tilting module.
    \item [(2)] $B=\textup{End}_A T$ is given by\[\xymatrix{
1\ar@<2pt>[r]^{\alpha}&2\ar@<2pt>[l]^{\beta}
}\] with $\beta\alpha=0$ and $\textup{gldim}B=2$.
    \item [(3)] Since $\mathcal{C}=\textup{Fac}T=\textup{Filt(Fac}\substack{1\\2})$, $\textup{dell}_T\substack{1\\2}=1$ and $2-\textup{dell}\ \textup{radHom}_A(T,T_2)=1$, one gets that
    $\textup{findim}B^{op}\leq\textup{dell}_T \mathcal{C}={\rm sup}\{\textup{dell}_T\substack{1\\2}=1, 2-\textup{dell}\ \textup{radHom}_A(T,T_2)+1\}=2$.
    \item [(4)] The projective dimension of $S(2)(={\rm top}\ \textup{Hom}_A(T, T_2))\in\textup{mod}B$ is $2$ and $\textup{findim}B^{op}\leq \textup{sup}\{2\ \textup{and}\ \textup{dell}_T\substack{1\\2}\}=2$
\end{itemize}
\end{exm}

We end this section with the following example which shows the finite projective dimension of the top in Theorem \ref{4.8} is not necessary.

\begin{exm}\label{4.11} Let $A$ be a preprojective algebra of type $A_3$. Then
\begin{itemize}
    \item [(1)] $T=\substack{1\\2\\3}\oplus\substack{1\\2}\oplus\substack{2\\1}$ is a $\tau$-tilting module with $B=\textup{End}_A T$ given by the quiver \[\xymatrix{3\ar@<2pt>[r]^b & 2\ar@<2pt>[l]^c \ar[r]^a & 1 }\] with the relations $bc=cb=0$. The ${\rm findim}B^{op}=1$.
    \item [(2)] $\mathcal{C}=\textup{Fac}T=T(S)=\textup{Filt}(\textup{Fac}(S_1\oplus S_3))$ with $S_1=\substack{1\\2\\3}$ and $S_3=\substack{\ \\2}$, $\textup{dell}_T S_1=\textup{dell}_T S_3=0$. The projective dimension of the top of $\textup{Hom}_A(T, T_2 )=S(2)$ is $\infty$ with the $2$-$\textup{dell}\Omega^1 S(2)$=$2$-$\textup{dell}S(3)=0$ since $S(3)\cong \Omega^2 S(3)$. Then $\textup{dell}_T\mathcal{C}=1$.
    \item [(3)] $U=\substack{2\\3}\oplus\substack{2\\1\ \  3\\2}\oplus\substack{2\\1}$ is a $\tau$-tilting module with $C=\textup{End}_A U$ given by\[\xymatrix{
1\ar@<2pt>[r]^{a} & 2\ar@<2pt>[l]^{d} \ar@<2pt>[r]^{b} & 3\ar@<2pt>[l]^c
}\] with $ab=cd=0, cb=ad=0, da=bc$. The ${\rm findim}C^{op}=1$.

\item [(4)] $\mathcal{C'}=\textup{Fac}U=T(S')=\textup{Filt}(\textup{Fac}S'_2)$ with $S'_2=\substack{2\\1\ \ 3}$. One gets that $\textup{dell}_U S'_2=0$, the projective dimension of the top of $\textup{Hom}_A(U, U_i )=S(i)$ is $\infty$ for $i=1, 3$, and the  $2$-$\textup{dell}\Omega^1 S(i)$=$2$-$\textup{dell}S(2)=0$ since $S(2)\cong \Omega^2 S(2)$. Then $\textup{dell}_U\mathcal{C'}=1$.
\end{itemize}
\end{exm}

\section{Applications to the finitistic dimension conjecture}

In this section, we apply Theorem \ref{4.5} to the finitistic dimension conjecture. Throughout this section, $T=\oplus_{i=1}^nT_i$ is a support $\tau$-tilting module with $T_i$ indecomposable and $\mathcal{C}=\text{Fac}T=T(S)=\text{Filt(Fac}S)$, where $S\simeq S_1\bigoplus \cdots \bigoplus S_m$ is the left finite semi-brick in \cite[Theorem 2.3(2)]{As}
 and $S_i$ is a brick. Denote by $B=\text{End}_A T$.

For a positive integer $i$, denote by $\Omega^i_T(\mathcal{C})$ the full subcategory of $\mathcal{C}$ formed by direct summands of $\Omega_T^n(X)\oplus T'$ for some object $X$ and $T'\in{\rm add}T$ and denote by $L_i$ the set of isomorphism classes of indecomposable modules in $\Omega^i_T(\mathcal{C})$. We have the following lemma.

\begin{lem}\label{5.1} Let $A$ be an algebra, $T\in {\rm mod}A$ a support $\tau$-tilting module and $\mathcal{C}=\textup{Fac}T$. If $L_1$ is a finite set, then global $2$-delooping level of $\mathcal{C}$ relative to $T$ is finite, that is, global-$2$-$\textup{dell}_T\mathcal{C}<\infty$.
\end{lem}

\begin{proof} It is obvious that $L_1\supseteq L_2\supseteq\cdots\supseteq L_n\supseteq\cdots$. Since there are finite number of elements in $L_1$, one gets there is an $n\geq 1$ such that $L_n=L_{n+1}$ and hence $L_{n+1}=L_{n+2}$ by the definition of $L_i$. So one gets that $L_n=L_{n+2}$. Then for any $M\in \mathcal{C}$, one gets an $N\in \mathcal{C}$ such that $\Omega^n_T(M)$ is a direct summand of $\Omega^{n+2}_T(N)$, that is, the $2$-delooping level of $M$ relative to $T$ is no more than $n$. The assertion holds.
\end{proof}

To show the main theorem in this section, we also need the following lemma.

\begin{lem}\label{5.2} Let $A$ be an algebra, $T\in {\rm mod}A$ a support $\tau$-tilting module and $\mathcal{C}=\textup{Fac}T$. If $L_1$ is a finite set, then the delooping level of $\mathcal{C}$ relative to $T$ is finite, that is, ${\rm dell}_T\mathcal{C}<\infty$.
\end{lem}

\begin{proof} It suffices to show that for a module $M$ in $\mathcal{C}$ the delooping level of $M$ relative to $T$ is no more than the $2$-delooping level of $M$ relative to $T$. By Lemma 5.1, we can assume that $2$-dell$_T M=n$, that is, $\Omega^n_T(M)$ is a direct summand of $\Omega^{n+2}_T(N)$ for some $N\in\mathcal{C}$. Then one gets that $\Omega^n_T(M)$ is a direct summand of $\Omega_T^{n+1}(\Omega^1_T N)$. So one gets the delooping level of $S_i$ is no more than $n$. By Definition \ref{4.1}, the assertion holds.
\end{proof}

  Now we are in a position to state the main result in this section.
  \begin{thm}\label{5.3} Let $A$ be an algebra, $T\in {\rm mod}A$ a support $\tau$-tilting module and $B={\rm End}_A T$. If ${\rm Fac}T\cap {\rm Sub}T$ is of finite representation type, then ${\rm findim}B^{op}<\infty$.
  \end{thm}

  \begin{proof} It is easy to show that ${\rm Fac}T\cap {\rm Sub}T={\rm add}L_1$. Since ${\rm Fac}T\cap {\rm Sub}T$ is of finite representation type, one gets that $L_1$ is a finite set. Then the assertion follows from Theorem \ref{4.5}, Lemma \ref{5.1} and Lemma \ref{5.2}.
\end{proof}

We have the following straight corollaries.

\begin{cor}\label{5.4} Let $A$ be an algebra, $T\in {\rm mod}A$ a support $\tau$-tilting module and $B={\rm End}_A T$. If ${\rm Fac}T$ or ${\rm Sub}T$ is of finite representation type, then ${\rm findim}B^{op}<\infty$.
\end{cor}

\begin{cor}\label{5.5} Let $A$ be an algebra of finite representation type, $T\in {\rm mod}A$ a support $\tau$-tilting module and $B={\rm End}_A T$. Then ${\rm findim}B^{op}<\infty$.
\end{cor}

In the rest of this section we give another hint on the finitistic dimension conjecture via support $\tau$-tilting modules.

\begin{prop}\label{5.6} Let $A$ be an algebra, $T\in {\rm mod}A$ a support $\tau$-tilting module and $B={\rm End}_A T$. Let $\bar{A}=A/{\rm ann}T$.
Then ${\rm findim}B^{op}<\infty$ if and only if so does $\bar{A}^{op}$.
\end{prop}
\begin{proof} By Proposition \ref{2.12}, $T$ is a tilting module in ${\rm mod}\bar{A}$. Since ${\rm Hom}_A(T, T)={\rm Hom}_{\bar{A}}(T, T)=B$, one gets that $T$ is also a tilting module in ${\rm mod}B^{op}$. Therefore one gets that $\bar{A}^{op}$ and $B^{op}$ are derived equivalent by \cite[Theorem 1.1]{R}. Then by \cite[Theorem 1.1]{PX}, the assertion holds.
\end{proof}

Recall that an algebra $A$ is called minimal representation infinite if it is of infinite representation type and every factor algebra of $A$ is of finite representation type. We have the following corollary.

\begin{cor}\label{5.7} Let $A$ be a minimal representation infinite algebra, $T\in {\rm mod}A$ a support $\tau$-tilting but not a tilting module but and $B={\rm End}_A T$. Then ${\rm findim}B^{op}<\infty$.
\end{cor}

\begin{proof} By Proposition \ref{2.12}, $T$ is a tilting module in ${\rm mod}\bar{A}$ and is also a tilting module in ${\rm mod}B^{op}$. Since $\bar{A}$ is of finite representation type, one gets that $\bar{A}^{op}$ is of finite representation type, and hence ${\rm findim}\bar {A}^{op}<\infty$. By Proposition \ref{5.6}, the assertion holds.
\end{proof}

We end this section with the following example.

\begin{exm}\label{5.8} Let $A$ be an algebra given by the quiver \[\xymatrix{
1\ar@<2pt>[r]^{\alpha}&2\ar@<2pt>[l]^{\beta}
}\]  with the relations $\alpha\beta=\beta\alpha=0$. Then
\begin{itemize}
\item[(1)] $A$ is a self-injective Nakayama algebra and ${\rm findim}A=0$.
\item[(2)] There are $5$ non-zero support $\tau$-tilting modules in ${\rm mod}A$ as follows:
$U_1=\substack{1\\2}\oplus\substack{2\\1}=A$, $U_2=\substack{2}\oplus\substack{2\\1}$, $U_3=\substack{1\\2}\oplus\substack{1}$, $U_4=\substack{2}$ and $U_5=\substack{1}$.
\item[(3)] Denote by $B_i={\rm End}_A U_i$ for $1\leq i\leq 5$. Then one gets that  ${\rm findim }B_1^{op}=0$, ${\rm findim} B_2^{op}=1$, ${\rm findim }B_3^{op}=1$, ${\rm findim }B_4^{op}=0$ and ${\rm findim} B_5^{op}=0$.
\end{itemize}
\end{exm}

\vskip 10pt

\noindent{\bf Acknowledgements}  Both of the authors are supported by the National Natural Science Foundation of China (Nos. 12171207 and  12371038). The authors also want to thank Professors Xiao-Wu Chen, Zhaoyong Huang, Zhi-Wei Li, Shengyong Pan and Zhibing Zhao for useful discussion and suggestions. They also thank Professor Lei Chen for the help with latex.

\bibliography{}

\begin{thebibliography}{9999}

\bibitem {A} {\sc T. Adachi}, {\em The classification of $\tau$-tilting modules over Nakayama
algebras}, J. Algebra, 452(2016), 227-262.

\bibitem{AIR} {\sc T. Adachi,  O. Iyama and I. Reiten}, {\em  $\tau$-tilting theory}, Compos. Math., {150} (3) (2014), 415-452.

\bibitem {AnMV} {\sc L. Angeleri H\"{u}gel, F. Marks and J. Vit\'oria}, {\em  Silting
modules}, Int. Math. Res. Not. IMRN, 4(2016), 1251-1284.


\bibitem{As}  {\sc S. Asai}, {\em Semibricks}, Int. Math. Res. Not. IMRN, 16(2020), 4993-5054.

\bibitem{ASS} {\sc I. Assem, D. Simson and A. Skowronski}, {\em Elements of the Representation Theory of Associative Algebras}, Volume 1. Techniques of Representation Theory, London Math. Soc. Student Texts {\bf 65}, Cambridge Univ. Press, 2006.
\bibitem{AuB}  {\sc  M. Auslander and D. A. Buchsbaum}, {\em Homological dimension in local rings}, Trans. Amer. Math. Soc., {85}(3)(1957), 390-405.
\bibitem{C} {\sc L. Chen}, {\em Derived equivalence and delooping level}, preprint at: arXiv: 2603.12974.
\bibitem{CLZZ} {\sc X.-W. Chen, Z.-W. Li, X. J. Zhang and Z. B. Zhao}, {\em Comparing $\tau$-tilting modules and 1-tilting modules}, preprint at: arXiv:2501.02466.
\bibitem{CH} {\sc Y. J. Chen and W. Hu}, {\em On the upper bound of finitistic dimension,} Master thesis in Beijing Normal University (in Chinese), 2025.
\bibitem {DIJ} {\sc L. Demonet, O. Iyama and G. Jasso}, {\em $\tau$-tilting finite algebras, bricks and $g$-vectors}, Int. Math. Res. Not. IMRN, 3(2019), 852-892.
\bibitem{G} {\sc V. Gelinas}, {\em The depth, the delooping level and the finitistic dimension}, Adv. Math., {394}   (2022), 108052.

\bibitem{GPS} {\sc E. L. Green, C. Psaroudakis and $\o$. Solberg}, {\em Reduction techniques for the finitistic dimension},  Trans. Amer. Math. Soc., 374(2021), no.10, 6839 - 6879.

\bibitem{Gu} {\sc R. Y. Guo}, {\em Symmetry of derived delooping level}, Algebr. Represent. Theory, 28(2025), no. 4, 1125-1137.

\bibitem{GuI} {\sc  R. Y. Guo and K. Igusa}, {\em Derived delooping levels and finitistic dimension}, Adv. Math., {464} (2025), 110152.

\bibitem{HR} {\sc D. Happel and C. M. Ringel}, {\em Tilted algebras}, Trans. Amer. Math. Soc., {274} (2) (1982), 399-443.

\bibitem{IT} {\sc K. Igusa and G. Todorov}, {\em On the finitistic global dimension conjecture for Artin algebras}, Inst. Commun.,45
American Mathematical Society,Providence, RI, 2005, 201-204.

\bibitem{IZ} {\sc O. Iyama and X. J. Zhang}, {\em Classifying $\tau$-tilting modules over the Auslander algebra of $K[x]/(x^{n})$ }, J. Math. Soc. Japan, 72(3) (2020), 731-764

\bibitem{KR}  {\sc L. Kershaw and J. Rickard},  {\em A finite dimensional algebra with infinite delooping level,}
Ann. Represent. Theory, 1(2024), no. 1, 61 - 65.

\bibitem{K} {\sc H. Krause}, {\em On the symmetry of the finitistic dimension}, C. R. Math. Acad. Sci. Paris, 361(2023), 1449-1453.

\bibitem{M} {\sc Y. Mizuno}, {\em  Classifying $\tau$-tilting modules over preprojective algebras
of Dynkin type,} Math. Zeit., {277}(3)(2014), 665-690.

\bibitem{PX}{\sc S. Y. Pan and C. C. Xi}, {\em Finiteness of finitistic dimension is invariant under derived equivalences},
J. Algebra, 322(2009), no. 1, 21-24.

\bibitem{R}{\sc J. Rickard}, {\em Derived equivalences as derived functors}, J. London Math. Soc., (2)43(1991), no. 1, 37-48.

\bibitem{Ri} {\sc C. M. Ringel}, {\em Brick chain filtrations}, preprint at: arXiv: 2411.18427.

\bibitem{S} {\sc E. Sen}, {\em Delooping level of Nakayama algebras,} Arch. Math. (Basel), 117(2021), no.2, 141-146.

\bibitem{Xi} {\sc C. C. Xi}, {\em The relative Auslander-Reiten theory of modules,} preprint at: \url{https://www.wemath.cn/~ccxi/}.

\bibitem{Xi2} {\sc C. C. Xi},  {\em On the finitistic dimension conjecture II. Related to finite global dimension}, Adv. Math., 201(2006), no.1, 116-142.

\bibitem{Z} {\sc X. J. Zhang}, {\em Self-orthogonal $\tau$-tilting modules and tilting modules,} J. Pure. Appl. Algebra, {226} (2022), 106860.

\end{thebibliography}

\vskip 10pt

{\footnotesize \noindent Mingfei Xu and Xiaojin Zhang\\
School of Mathematics and Statistics, Jiangsu Normal University, Xuzhou 221116, Jiangsu, PR China.\\}



\end{document}